\documentclass[a4paper,10pt]{amsart}

\usepackage{amsmath,amssymb,amsfonts,latexsym,bm}

\newtheorem{theorem}{Theorem}[section]
\newtheorem{lemma}[theorem]{Lemma}
\newtheorem{corollary}[theorem]{Corollary}

\begin{document}

\title[Generalized Hybrid Fibonacci Numbers and Their Properties]{Investigation of Generalized Hybrid Fibonacci Numbers and Their Properties}

\author[G. Cerda-Morales]{Gamaliel Cerda-Morales}
\address{Instituto de Matem\'aticas, Pontificia Universidad Cat\'olica de Valpara\'iso, Blanco Viel 596, Valpara\'iso, Chile.}
\email{gamaliel.cerda.m@mail.pucv.cl}


\begin{abstract}
In \cite{Oz}, M. \"Ozdemir defined a new non-commutative number system called hybrid numbers. In this paper, we define the hybrid Fibonacci and Lucas numbers. This number system can be accepted as a generalization of the complex ($\textbf{i}^{2}=-1$), hyperbolic ($\textbf{h}^{2}=1$) and dual Fibonacci number ($\bm{\varepsilon}^{2}=0$) systems. Furthermore, a hybrid Fibonacci number is a number created with any combination of the complex, hyperbolic and dual numbers satisfying the relation $\textbf{ih}=-\textbf{hi}=\bm{\varepsilon}+\textbf{i}$. Then we used the Binet's formula to show some properties of the hybrid Fibonacci numbers. We get some generalized identities of the hybrid Fibonacci and hybrid Lucas numbers.

\vspace{2mm}


\vspace{2mm}


\end{abstract}


\maketitle


\section{Introduction}
\setcounter{equation}{0}
The most famous generalization of complex numbers is quaternions. In 1843, William Rowan Hamilton described the set of quaternions $$\mathbb{H}=\{a+b\textbf{i}+c\textbf{j}+d\textbf{k}:\ \textbf{i}^{2}=\textbf{j}^{2}=\textbf{k}^{2}=\textbf{i}\textbf{j}\textbf{k}=-1\}$$ and James Cockle defined coquaternions (split quaternions) $$\overline{\mathbb{H}}=\{a+b\textbf{i}+c\textbf{j}+d\textbf{k}:\ \textbf{i}^{2}=-1,\ \textbf{j}^{2}=\textbf{k}^{2}=\textbf{i}\textbf{j}\textbf{k}=1\}$$ in 1849 (see \cite{Co}). Quaternions and coquaternions are used to define 3D Euclidean and Lorentzian rotations, respectively. A set of split quaternions is non-commutative and contains zero divisors, nilpotent elements, and nontrivial idempotents (see \cite{Ku-Ya,Oz-Er}). Previous studies have examined the geometric and physical applications of split quaternions, which are required in solving split quaternionic equations \cite{Fre-Li}.

In particular, Fibonacci and Lucas quaternions cover a wide range of interest in modern mathematics as they appear in the comprehensive works of \cite{Hor1,Hor2}. For example, the Fibonacci quaternion denoted by $Q_{F,n}$, is the $n$-th term of the sequence where each term is the sum of the two previous terms beginning with the initial values $Q_{F,0}=\textbf{i}+\textbf{j}+2\textbf{k}$ and $Q_{F,1}=1+\textbf{i}+2\textbf{j}+3\textbf{k}$. The well-known Fibonacci quaternion $Q_{F,n}$ is defined as
\begin{equation}\label{Hor}
Q_{F,n}=F_{n}+\textbf{i}F_{n+1}+\textbf{j}F_{n+2}+\textbf{k}F_{n+3}
\end{equation} 
and the Lucas quaternion is defined as $Q_{L,n}=L_{n}+\textbf{i}L_{n+1}+\textbf{j}L_{n+2}+\textbf{k}L_{n+3}$ for $n\geq0$, where $F_{n}$ and $L_{n}$ are $n$-th Fibonacci and Lucas number, respectively.

Ipek \cite{Ipe} studied the $(p,q)$-Fibonacci quaternions $Q_{\mathcal{F},n}$ which is defined as 
\begin{equation}\label{Ipe}
Q_{\mathcal{F},n}=pQ_{\mathcal{F},n-1}+qQ_{\mathcal{F},n-2},\ n\geq2
\end{equation}
with initial conditions $Q_{\mathcal{F},0}=\textbf{i}+p\textbf{j}+(p^{2}+q)\textbf{k}$, $Q_{\mathcal{F},1}=1+p\textbf{i}+(p^{2}+q)\textbf{j}+(p^{3}+2pq)\textbf{k}$ and $p^{2}+4q>0$. If $p=q=1$, we get the classical Fibonacci quaternion $Q_{F,n}$ \cite{Hal1}. If $p=2q=2$, we get the Pell quaternion $Q_{P,n}=P_{n}+\textbf{i}P_{n+1}+\textbf{j}P_{n+2}+\textbf{k}P_{n+3}$ (see \cite{Ci-Ipe}), where $P_{n}$ is the $n$-th Pell number.

The well-known Binet's formulas for $(p,q)$-Fibonacci quaternion and $(p,q)$-Lucas quaternion, see \cite{Ipe}, are given by
\begin{equation}\label{Binet}
Q_{\mathcal{F},n}=\frac{\underline{\alpha}\alpha^{n}-\underline{\beta}\beta^{n}}{\alpha-\beta}\ \textrm{and}\ Q_{\mathcal{L},n}=\underline{\alpha}\alpha^{n}+\underline{\beta}\beta^{n},
\end{equation}
where $\alpha,\beta$ are roots of the characteristic equation $t^{2}-pt-q=0$, and $\underline{\alpha}=1+\alpha \textbf{i}+\alpha^{2} \textbf{j}+\alpha^{3} \textbf{k}$ and $\underline{\beta}=1+\beta \textbf{i}+\beta^{2} \textbf{j}+\beta^{3} \textbf{k}$. We note that $\alpha+\beta=p$, $\alpha \beta=-q$ and $\alpha-\beta=\sqrt{p^{2}+4q}$.

The generalized of Fibonacci quaternion $Q_{w,n}$ is defined recently by Halici and Karata\c{s} in \cite{Hal3} as $Q_{w,0}=a+b\textbf{i}+(pb+qa)\textbf{j}+((p^{2}+q)b+pqa)\textbf{k}$, $Q_{w,1}=b+(pb+qa)\textbf{i}+((p^{2}+q)b+pqa)\textbf{j}+((p^{3}+2pq)b+q(p^{2}+q)a)\textbf{k}$ and $Q_{w,n}=pQ_{w,n-1}+qQ_{w,n-2}$, for $n\geq2$ which we call the generalized Fibonacci or Horadam quaternions. So, each term of the generalized Fibonacci sequence $\{Q_{w,n}\}_{n\geq0}$ is called generalized Fibonacci quaternion. 

The Binet formula for generalized Fibonacci quaternion $Q_{w,n}$, see \cite{Hal3}, is given by
\begin{equation}\label{Binet2}
Q_{w,n}=\frac{A\underline{\alpha}\alpha^{n}-B\underline{\beta}\beta^{n}}{\alpha-\beta},
\end{equation}
where $A=b-a\beta$, $B=b-a\alpha$, and $\alpha,\beta$ are roots of the characteristic equation $t^{2}-pt-q=0$, $\underline{\alpha}=1+\alpha \textbf{i}+\alpha^{2} \textbf{j}+\alpha^{3} \textbf{k}$ and $\underline{\beta}=1+\beta \textbf{i}+\beta^{2} \textbf{j}+\beta^{3} \textbf{k}$. If $a=0$ and $b=1$, we get the classical $(p,q)$-Fibonacci quaternion $Q_{\mathcal{F},n}$. If $a=2$ and $b=p$, we get the $(p,q)$-Lucas quaternion $Q_{\mathcal{L},n}$.

On the other hand, Olariu \cite{Ola} defined a different generalization of $n$-dimensional complex numbers naming them twocomplex numbers, threecomplex numbers. Olariu used the name twocomplex numbers for hyperbolic numbers. He studied the geometrical and the algebraic properties of these numbers. For example, the set of threecomplex numbers was defined as $$\mathbb{C}_{3}=\{a+b\textbf{h}+c\textbf{k}:\ a,b,c\in \mathbb{R},\ \textbf{h}^{2}=\textbf{k},\ \textbf{k}^{2}=\textbf{h},\ \textbf{hk}=1\}.$$ In 2004, Anthony Harkin and Joseph Harkin \cite{Har} generalized two dimensional complex numbers as $\mathbb{C}_{p}=\{\textbf{z}=a+b\textbf{i}:\ a,b\in \mathbb{R},\ \textbf{i}^{2}=p\}$. They gave some trigonometric relations for this generalization. After, Catoni et al. \cite{Ca} defined two dimensional hypercomplex numbers as
$$\mathbb{C}_{\alpha,\beta}=\{\textbf{z}=x+y\textbf{i}:\ x,y\in \mathbb{R},\ \textbf{i}^{2}=\alpha +\textbf{i}\beta\}$$
and extended the relationship between these numbers and Euclidean and semi-Euclidean geometry. Furthermore, this generalization is also expressible as a quotient ring $\mathbb{R}[x]/ (x^{2}-\beta x-\alpha )$. 

In 2017, Zaripov \cite{Za} presented a theory of commutative two-dimensional conformal hyperbolic numbers as a generalization of the theory of hyperbolic numbers. Recently, \"Ozdemir \cite{Oz} defined a new generalization of complex, hyperbolic and dual numbers different from above generalizations. In this generalization, the author gave a system of such numbers that consists of all three number systems together. This set was called hybrid numbers, denoted by $\mathbb{K}$, is defined as
\begin{equation}\label{Hy}
\mathbb{K}=\left\lbrace \textbf{z}=a+b\textbf{i}+c\bm{\varepsilon}+d\textbf{h}: a,b,c,d\in \mathbb{R},\ \begin{array}{c} \textbf{i}^{2}=-1,\ \bm{\varepsilon}^{2}=0,\ \textbf{h}^{2}=1,\\ 
 \textbf{ih}=-\textbf{hi}=\bm{\varepsilon}+\textbf{i} \end{array}\right\rbrace.
\end{equation}
Two hybrid numbers are equal if all their components are equal, one by one. The sum of two hybrid numbers is defined by summing their components. Addition operation in the hybrid numbers is both commutative and associative. Zero is the null element. With respect to the addition operation, the inverse element of $\textbf{z}$ is $-\textbf{z}$, which is defined as having all the components of $\textbf{z}$ changed in their signs. This implies that, $(\mathbb{K}, +)$ is an Abelian group.

The hybridian product is obtained by distributing the terms on the right as in ordinary algebra, preserving that the multiplication order of the units and then writing the values of followings replacing each product of units by the equalities $\textbf{i}^{2}=-1,\ \bm{\varepsilon}^{2}=0,\ \textbf{h}^{2}=1$ and $\textbf{ih}=-\textbf{hi}=\bm{\varepsilon}+\textbf{i}$. Using these equalities we can find the product of any two hybrid units. For example, let's find $\textbf{i}\bm{\varepsilon}$. For this, let's multiply $\textbf{ih}=\bm{\varepsilon}+\textbf{i}$ by $\textbf{i}$ from the left. Thus, we get $\textbf{i}\bm{\varepsilon}=1-\textbf{h}$. If we continue in a similar way, we get the following multiplication table.

\begin{table}[ht] 
\caption{The multiplication table for the basis of $\mathbb{K}$.} 
\centering      
\begin{tabular}{lllll}
\hline
$\times $ & $1$ & $\textbf{i}$ & $\bm{\varepsilon}$ & $\textbf{h}$  \\ \hline
$1$ & $1$ & $\textbf{i}$ & $\bm{\varepsilon}$ & $\textbf{h}$  
\\ 
$\textbf{i}$ & $\textbf{i}$ & $-1$ & $1-\textbf{h}$ & $\bm{\varepsilon}+\textbf{i}$
\\ 
$\bm{\varepsilon}$ & $\bm{\varepsilon}$ & $1+\textbf{h}$ & $0$ & $-\bm{\varepsilon}$
\\ 
$\textbf{h}$ & $\textbf{h}$ & $-(\bm{\varepsilon}+\textbf{i})$ & $\bm{\varepsilon}$ & $1$ 
\\ \hline
\end{tabular}
\label{table:1}  
\end{table}
The table \ref{table:1} shows us that the multiplication operation in the hybrid numbers is not commutative. But it has the property of associativity. The conjugate of a hybrid number $\textbf{z}=a+b\textbf{i}+c\bm{\varepsilon}+d\textbf{h}$, denoted by $\overline{\textbf{z}}$, is defined as $\textbf{z}=a-b\textbf{i}-c\bm{\varepsilon}-d\textbf{h}$ as in the quaternions. The conjugate of the sum of hybrid numbers is equal to the sum of their conjugates. Also, according to the hybridian product, we have $\textbf{z}\overline{\textbf{z}}=\overline{\textbf{z}}\textbf{z}$. The real number $$\mathcal{C}(\textbf{z})=\textbf{z}\overline{\textbf{z}}=\overline{\textbf{z}}\textbf{z}=a^{2}+(b-c)^{2}-c^{2}-d^{2}$$ is called the character of the hybrid number $\textbf{z}=a+b\textbf{i}+c\bm{\varepsilon}+d\textbf{h}$. The real number  $\sqrt{\mathcal{C}(\textbf{z})}$ will be called the norm of the hybrid number $\textbf{z}$ and will be denoted by $\lVert \textbf{z}\rVert_{\mathbb{K}}$.

In this study, we define the hybrid Fibonacci and hybrid Lucas numbers. We give the generating functions and Binet formulas for these numbers. Moreover, the well-known properties e.g. Cassini and Catalan identities have been obtained for these numbers.

\section{Generalized Hybrid Fibonacci and Lucas Numbers}
\setcounter{equation}{0}
We define the $n$-th hybrid $(p,q)$-Fibonacci and hybrid $(p,q)$-Lucas numbers, respectively, by the following recurrence relations
\begin{equation}\label{h:f}
\mathbb{H}\mathcal{F}_{n}=\mathcal{F}_{n}+\mathcal{F}_{n+1}\textbf{i}+\mathcal{F}_{n+2}\bm{\varepsilon}+\mathcal{F}_{n+3}\textbf{h}
\end{equation}
and
\begin{equation}\label{h:l}
\mathbb{H}\mathcal{L}_{n}=\mathcal{L}_{n}+\mathcal{L}_{n+1}\textbf{i}+\mathcal{L}_{n+2}\bm{\varepsilon}+\mathcal{L}_{n+3}\textbf{h},
\end{equation}
where $\mathcal{F}_{n}$ and $\mathcal{L}_{n}$ are the $n$-th $(p,q)$-Fibonacci number and $(p,q)$-Lucas number, respectively. Here $\{\textbf{i}, \bm{\varepsilon}, \textbf{h}\}$ satisfies the multiplication rule given in the Table \ref{table:1}.

By some elementary calculations we find the following recurrence relations for the generalized hybrid Fibonacci and Lucas numbers from (\ref{h:f}) and (\ref{h:l}):
\begin{equation}
\begin{aligned}
p\mathbb{H}\mathcal{F}_{n}+q\mathbb{H}\mathcal{F}_{n-1}&=p(\mathcal{F}_{n}+\mathcal{F}_{n+1}\textbf{i}+\mathcal{F}_{n+2}\bm{\varepsilon}+\mathcal{F}_{n+3}\textbf{h})\\
&\ \ +q(\mathcal{F}_{n-1}+\mathcal{F}_{n}\textbf{i}+\mathcal{F}_{n+1}\bm{\varepsilon}+\mathcal{F}_{n+2}\textbf{h})\\
&=(p\mathcal{F}_{n}+q\mathcal{F}_{n-1})+(p\mathcal{F}_{n+1}+q\mathcal{F}_{n})\textbf{i}+(p\mathcal{F}_{n+2}+q\mathcal{F}_{n+1})\bm{\varepsilon}\\
&\ \ +(p\mathcal{F}_{n+3}+q\mathcal{F}_{n+2})\textbf{h}\\
&=\mathcal{F}_{n+1}+\mathcal{F}_{n+2}\textbf{i}+\mathcal{F}_{n+3}\bm{\varepsilon}+\mathcal{F}_{n+4}\textbf{h}\\
&=\mathbb{H}\mathcal{F}_{n+1}
\end{aligned} \label{sum}
\end{equation}
and similarly $\mathbb{H}\mathcal{L}_{n+1}=p\mathbb{H}\mathcal{L}_{n}+q\mathbb{H}\mathcal{L}_{n-1}$, for $n\geq1$.

In this paper, following Halici and Karata\c{s} \cite{Hal3}, we define the generalized hybrid Fibonacci numbers as
\begin{equation}\label{h:g}
\mathbb{H}\mathcal{J}_{n}=p\mathbb{H}\mathcal{J}_{n-1}+q\mathbb{H}\mathcal{J}_{n-2},\ n\geq 2,
\end{equation}
where $\mathbb{H}\mathcal{J}_{0}=a+b\textbf{i}+(pb+qa)\bm{\varepsilon}+((p^{2}+q)b+pqa)\textbf{h}$ and $\mathbb{H}\mathcal{J}_{1}=b+(pb+qa)\textbf{i}+((p^{2}+q)b+pqa)\bm{\varepsilon}+((p^{3}+2pq)b+q(p^{2}+q)a)\textbf{h}$.

So, each term of the generalized hybrid Fibonacci sequence $\{\mathbb{H}\mathcal{J}_{n}\}_{n\geq0}$ is called generalized hybrid Fibonacci number. Furthermore, if $a=0$ and $b=1$, we get the hybrid $(p,q)$-Fibonacci number $\mathbb{H}\mathcal{F}_{n}$. If $a=2$ and $b=p$, we get the hybrid $(p,q)$-Lucas number $\mathbb{H}\mathcal{L}_{n}$.

Generating functions for the generalized hybrid Fibonacci numbers are given in the next theorem.
\begin{theorem}
The generating function for the generalized hybrid Fibonacci number is
\begin{equation}\label{gg}
\sum_{r=0}^{\infty}\mathbb{H}\mathcal{J}_{r}t^{r}=\frac{\left\lbrace \begin{array}{c} a+b\textbf{i}+(pb+qa)\bm{\varepsilon}+((p^{2}+q)b+pqa)\textbf{h}\\ 
+t((b-pa)+qa\textbf{i}+qb\bm{\varepsilon}+(pqb+q^{2}a)\textbf{h})\end{array}\right\rbrace}{1-pt-qt^2}.
\end{equation}
\end{theorem}
\begin{proof}
Assuming that the generating function of the hybrid number $\{\mathbb{H}\mathcal{J}_{n}\}$ has the form $G(t)=\sum_{r=0}^{\infty}\mathbb{H}\mathcal{J}_{r}t^{r}$, we obtain that
\begin{align*}
(1-pt-qt^2)G(t)&=(\mathbb{H}\mathcal{J}_{0}+\mathbb{H}\mathcal{J}_{1}t+ \cdots ) -p(\mathbb{H}\mathcal{J}_{0}t+\mathbb{H}\mathcal{J}_{1}t^{2}+ \cdots ) \\
&\ \ -q(\mathbb{H}\mathcal{J}_{0}t^{2}+\mathbb{H}\mathcal{J}_{1}t^{3}+ \cdots ) \\
&=\mathbb{H}\mathcal{J}_{0}+t(\mathbb{H}\mathcal{J}_{1}-p\mathbb{H}\mathcal{J}_{0}),
\end{align*}
since $\mathbb{H}\mathcal{J}_{n}=p\mathbb{H}\mathcal{J}_{n-1}+q\mathbb{H}\mathcal{J}_{n-2}$, $n\geq 2$ and the coefficients of $t^{n}$ for $n\geq 2$ are equal to zero. In equivalent form is
\begin{align*}
\sum_{r=0}^{\infty}\mathbb{H}\mathcal{J}_{r}t^{r}&=\frac{\mathbb{H}\mathcal{J}_{0}+t(\mathbb{H}\mathcal{J}_{1}-p\mathbb{H}\mathcal{J}_{0})}{1-pt-qt^2}\\
&=\frac{\left\lbrace \begin{array}{c} a+b\textbf{i}+(pb+qa)\bm{\varepsilon}+((p^{2}+q)b+pqa)\textbf{h}\\ 
+t((b-pa)+qa\textbf{i}+qb\bm{\varepsilon}+(pqb+q^{2}a)\textbf{h})\end{array}\right\rbrace}{1-pt-qt^2}.
\end{align*}
So, the theorem is proved.
\end{proof}

The next theorem gives the Binet formulas for the generalized hybrid Fibonacci numbers.
\begin{theorem}
For any integer $n\geq 0$, the $n$-th generalized Fibonacci number is
\begin{equation}\label{binH}
\mathbb{H}\mathcal{J}_{n}=\frac{A\underline{\bm{\alpha}}\alpha^{n}-B\underline{\bm{\beta}}\beta^{n}}{\alpha-\beta},
\end{equation}
where $A=b-a\beta$, $B=b-a\alpha$, and $\alpha,\beta$ are roots of the characteristic equation $t^{2}-pt-q=0$, $\underline{\bm{\alpha}}=1+\alpha \textbf{i}+\alpha^{2} \bm{\varepsilon}+\alpha^{3} \textbf{h}$ and $\underline{\bm{\beta}}=1+\beta \textbf{i}+\beta^{2} \bm{\varepsilon}+\beta^{3} \textbf{h}$. If $a=0$ and $b=1$, we get the hybrid $(p,q)$-Fibonacci number $\mathbb{H}\mathcal{F}_{n}$. If $a=2$ and $b=p$, we get the hybrid $(p,q)$-Lucas number $\mathbb{H}\mathcal{L}_{n}$.
\end{theorem}
\begin{proof}
For the Eq. (\ref{binH}), we have
\begin{align*}
\alpha \mathbb{H}\mathcal{J}_{n+1}+q\mathbb{H}\mathcal{J}_{n}&=\alpha(\mathcal{J}_{n+1}+\mathcal{J}_{n+2}\textbf{i}+\mathcal{J}_{n+3} \bm{\varepsilon}+\mathcal{J}_{n+4}\textbf{h})\\
&\ \ +q(\mathcal{J}_{n}+\mathcal{J}_{n+1} \textbf{i}+\mathcal{J}_{n+2}\bm{\varepsilon}+\mathcal{J}_{n+3}\textbf{h})\\
&=(\alpha \mathcal{J}_{n+1} +q\mathcal{J}_{n})+(\alpha \mathcal{J}_{n+2} +q\mathcal{J}_{n+1})\textbf{i}+(\alpha \mathcal{J}_{n+3} +q\mathcal{J}_{n+2}) \bm{\varepsilon}\\
&\ \ +(\alpha \mathcal{J}_{n+4} +q\mathcal{J}_{n+3})\textbf{h}.
\end{align*}
From the identity $\alpha \mathcal{J}_{n+1} +q\mathcal{J}_{n}=\alpha^{n}(\alpha b +qa)$, we obtain
\begin{equation}\label{ge1}
\alpha \mathbb{H}\mathcal{J}_{n+1} +q\mathbb{H}\mathcal{J}_{n}=\underline{\bm{\alpha}}\alpha^{n}(\alpha b +qa).
\end{equation}
Similarly, we have
\begin{equation}\label{ge2}
\beta \mathbb{H}\mathcal{J}_{n+1} +q\mathbb{H}\mathcal{J}_{n}=\underline{\bm{\beta}}\beta^{n}(\beta b +qa).
\end{equation}
Subtracting Eq. (\ref{ge2}) from Eq. (\ref{ge1}) gives $$(\alpha-\beta)\mathbb{H}\mathcal{J}_{n+1} =A\underline{\bm{\alpha}}\alpha^{n+1}-B\underline{\bm{\beta}}\beta^{n+1},$$ where $A=b-a\beta$, $B=b-a\alpha$ and $\alpha,\beta$ are roots of the characteristic equation $t^{2}-pt-q=0$. Furthermore, $\underline{\bm{\alpha}}=1+\alpha \textbf{i}+\alpha^{2} \bm{\varepsilon}+\alpha^{3} \textbf{h}$ and $\underline{\bm{\beta}}=1+\beta \textbf{i}+\beta^{2} \bm{\varepsilon}+\beta^{3} \textbf{h}$. So, the theorem is proved.
\end{proof}

There are three well-known identities for generalized Fibonacci numbers, namely, Catalan's, Cassini's, and d'Ocagne's identities. The proofs of these identities are based on Binet formulas. We can obtain these types of identities for generalized hybrid Fibonacci numbers using the Binet formula for $\mathbb{H}\mathcal{J}_{n}$. Then, we require $\underline{\bm{\alpha}}\underline{\bm{\beta}}$ and $\underline{\bm{\beta}}\underline{\bm{\alpha}}$. These products are given in the next lemma.
\begin{lemma}
We have
\begin{equation}\label{eq:1}
\underline{\bm{\alpha}}\underline{\bm{\beta}}=\mathbb{H}\mathcal{L}_{0}-(q^{3}+pq-q+1)+q(\alpha-\beta)(\mathbb{H}\mathcal{F}_{0}-\bm{\omega}),
\end{equation}
and
\begin{equation}\label{eq:2}
\underline{\bm{\beta}}\underline{\bm{\alpha}}=\mathbb{H}\mathcal{L}_{0}-(q^{3}+pq-q+1)-q(\alpha-\beta)(\mathbb{H}\mathcal{F}_{0}-\bm{\omega}),
\end{equation}
where $\bm{\omega}=(1-p)\textbf{i}-q\bm{\varepsilon}+(p^{2}+q+1)\textbf{h}$ and $\alpha-\beta =\sqrt{p^{2}+4q}$.
\end{lemma}
\begin{proof}
From the definitions of $\underline{\bm{\alpha}}$ and $\underline{\bm{\beta}}$, and using $\textbf{i}^{2}=-1,\ \bm{\varepsilon}^{2}=0,\ \textbf{h}^{2}=1$ and $ \textbf{ih}=-\textbf{hi}=\bm{\varepsilon}+\textbf{i}$ in Table \ref{table:1}, we have
\begin{align*}
\underline{\bm{\alpha}}\underline{\bm{\beta}}&=(1+\alpha \textbf{i}+\alpha^{2} \bm{\varepsilon}+\alpha^{3} \textbf{h})(1+\beta \textbf{i}+\beta^{2} \bm{\varepsilon}+\beta^{3} \textbf{h})\\
&=2+(\alpha+\beta)\textbf{i}+(\alpha^{2}+\beta^{2})\bm{\varepsilon}+(\alpha^{3}+\beta^{3})\textbf{h}-1+\alpha\beta(-1+\alpha+\beta+\alpha^{2}\beta^{2})\\
&\ \ -\alpha\beta (\alpha^{2}-\beta^{2})\textbf{i}-\alpha\beta (\alpha^{2}-\beta^{2}-\alpha^{2}\beta+\alpha\beta^{2})\bm{\varepsilon}+\alpha\beta(\alpha-\beta)\textbf{h}\\
&=2+p\textbf{i}+(p^{2}+2q)\bm{\varepsilon}+(p^{3}+3pq)\textbf{h}-(q^{3}+pq-q+1)\\
&\ \ +q(\alpha-\beta)(p\textbf{i}+(p+q)\bm{\varepsilon}-\textbf{h})\\
&=\mathbb{H}\mathcal{L}_{0}-(q^{3}+pq-q+1)+q(\alpha-\beta)(p\textbf{i}+(p+q)\bm{\varepsilon}-\textbf{h})\\
&=\mathbb{H}\mathcal{L}_{0}-(q^{3}+pq-q+1)+q(\alpha-\beta)(\mathbb{H}\mathcal{F}_{0}-\bm{\omega}),
\end{align*}
where $\bm{\omega}=(1-p)\textbf{i}-q\bm{\varepsilon}+(p^{2}+q+1)\textbf{h}$ and the final equation gives Eq. (\ref{eq:1}). The other identity can be computed similarly.
\end{proof}

This lemma gives us the following useful identity:
\begin{equation}\label{eq:3}
\underline{\bm{\alpha}}\underline{\bm{\beta}}+\underline{\bm{\beta}}\underline{\bm{\alpha}}=2(\mathbb{H}\mathcal{L}_{0}-(q^{3}+pq-q+1)).
\end{equation}

Catalan's identities for generalized Fibonacci quaternions are given in the next theorem.
\begin{theorem}\label{teO}
For any integers $m\geq r \geq 0$, we have
\begin{equation}\label{eq:4}
\mathbb{H}\mathcal{J}_{m}^{2}-\mathbb{H}\mathcal{J}_{m+r}\mathbb{H}\mathcal{J}_{m-r}=-AB(-q)^{m}\mathcal{F}_{-r}\left\lbrace \begin{array}{c} (\mathbb{H}\mathcal{L}_{0}-(q^{3}+pq-q+1))\mathcal{F}_{r}\\
+q(\mathbb{H}\mathcal{F}_{0}-\bm{\omega})\mathcal{L}_{r}\end{array}\right\rbrace,
\end{equation}
where $A=b-a\beta$, $B=b-a\alpha$, $\bm{\omega}=(1-p)\textbf{i}-q\bm{\varepsilon}+(p^{2}+q+1)\textbf{h}$ and $\mathcal{F}_{r}$, $\mathcal{L}_{r}$ are the $r$-th $(p,q)$-Fibonacci and $(p,q)$-Lucas numbers, respectively.
\end{theorem}
\begin{proof}
From the Binet formula for generalized Fibonacci quaternions $\mathbb{H}\mathcal{J}_{m}$ in (\ref{binH}) and $(\alpha-\beta)^{2}=p^{2}+4q$, we have
\begin{align*}
(p^{2}+4q)&\left(\mathbb{H}\mathcal{J}_{m}^{2}-\mathbb{H}\mathcal{J}_{m+r}\mathbb{H}\mathcal{J}_{m-r}\right)\\
&=\left(A\underline{\bm{\alpha}}\alpha^{m}-B\underline{\bm{\beta}}\beta^{m}\right)^{2}- \left(A\underline{\bm{\alpha}}\alpha^{m+r}-B\underline{\bm{\beta}}\beta^{m+r}\right)\left(A\underline{\bm{\alpha}}\alpha^{m-r}-B\underline{\bm{\beta}}\beta^{m-r}\right)\\
&=AB(-q)^{m-r}\big(\underline{\bm{\alpha}}\underline{\bm{\beta}}\alpha^{2r}+\underline{\bm{\beta}}\underline{\bm{\alpha}}\beta^{2r}-(-q)^{r}\left(\underline{\bm{\alpha}}\underline{\bm{\beta}}+\underline{\bm{\beta}}\underline{\bm{\alpha}}\right)\big).
\end{align*}
We require Eqs. (\ref{eq:1}) and (\ref{eq:2}). Using this equations, we obtain
\begin{align*}
\mathbb{H}\mathcal{J}_{m}^{2}&-\mathbb{H}\mathcal{J}_{m+r}\mathbb{H}\mathcal{J}_{m-r}\\
&=\frac{AB(-q)^{m-r}}{p^{2}+4q}\left\lbrace\begin{array}{c} (\mathbb{H}\mathcal{L}_{0}-(q^{3}+pq-q+1))(\alpha^{2r}+\beta^{2r}-2(-q)^{r})\\
+q(\alpha-\beta)(\mathbb{H}\mathcal{F}_{0}-\bm{\omega})(\alpha^{2r}-\beta^{2r})\end{array}\right\rbrace\\
&=\frac{AB(-q)^{m-r}}{p^{2}+4q}\left\lbrace \begin{array}{c} (\mathbb{H}\mathcal{L}_{0}-(q^{3}+pq-q+1))(\mathcal{L}_{2r}-2(-q)^{r})\\
+q(p^{2}+4q)(\mathbb{H}\mathcal{F}_{0}-\bm{\omega})\mathcal{F}_{2r}\end{array}\right\rbrace.
\end{align*}
Using the identity $(p^{2}+4q)\mathcal{F}_{r}^{2}=\mathcal{L}_{2r}-2(-q)^{r}$ gives
$$\mathbb{H}\mathcal{J}_{m}^{2}-\mathbb{H}\mathcal{J}_{m+r}\mathbb{H}\mathcal{J}_{m-r}=AB(-q)^{m-r}\left\lbrace \begin{array}{c} (\mathbb{H}\mathcal{L}_{0}-(q^{3}+pq-q+1))\mathcal{F}_{r}^{2}\\
+q(\mathbb{H}\mathcal{F}_{0}-\bm{\omega})\mathcal{F}_{2r}\end{array}\right\rbrace,$$
where $\mathcal{L}_{r}$, $\mathcal{F}_{r}$ are the $r$-th $(p,q)$-Lucas and $(p,q)$-Fibonacci numbers, respectively. With the help of the identities $\mathcal{F}_{2r}=\mathcal{F}_{r}\mathcal{L}_{r}$ and $\mathcal{F}_{-r}=-(-q)^{-r}\mathcal{F}_{r}$, we have Eq. (\ref{eq:4}). The proof is completed.
\end{proof}

Taking $r=1$ in the Theorem \ref{teO} and using the identity $\mathcal{F}_{-1}=\frac{1}{q}$, we obtain Cassini's identities for generalized Fibonacci quaternions.
\begin{corollary}
For any integer $m$, we have
\begin{equation}\label{eq:5}
\mathbb{H}\mathcal{J}_{m}^{2}-\mathbb{H}\mathcal{J}_{m+1}\mathbb{H}\mathcal{J}_{m-1}=AB(-q)^{m-1}\left\lbrace \begin{array}{c} (\mathbb{H}\mathcal{L}_{0}-(q^{3}+pq-q+1))\\
+pq(\mathbb{H}\mathcal{F}_{0}-\bm{\omega})\end{array} \right\rbrace,
\end{equation}
where $A=b-a\beta$, $B=b-a\alpha$ and $\bm{\omega}=(1-p)\textbf{i}-q\bm{\varepsilon}+(p^{2}+q+1)\textbf{h}$.
\end{corollary}

The following theorem gives d'Ocagne's identities for generalized hybrid Fibonacci numbers.
\begin{theorem}\label{teoD}
For any integers $r$ and $m$, we have
\begin{equation}\label{eq:6}
\mathbb{H}\mathcal{J}_{r}\mathbb{H}\mathcal{J}_{m+1}-\mathbb{H}\mathcal{J}_{r+1}\mathbb{H}\mathcal{J}_{m}=(-q)^{m}AB \left\lbrace \begin{array}{c} (\mathbb{H}\mathcal{L}_{0}-(q^{3}+pq-q+1))\mathcal{F}_{r-m}\\
+q(\mathbb{H}\mathcal{F}_{0}-\bm{\omega}) \mathcal{L}_{r-m}  \end{array} \right\rbrace
\end{equation}
$\mathcal{F}_{r}$, $\mathcal{L}_{r}$ are the $r$-th $(p,q)$-Fibonacci and $(p,q)$-Lucas numbers, respectively.
\end{theorem}
\begin{proof}
Using the Binet formula for the generalized hybrid Fibonacci numbers gives
\begin{align*}
(p^{2}+4q)(\mathbb{H}\mathcal{J}_{r}\mathbb{H}\mathcal{J}_{m+1}&-\mathbb{H}\mathcal{J}_{r+1}\mathbb{H}\mathcal{J}_{m})\\
&=\left(A\underline{\bm{\alpha}}\alpha^{r}-B\underline{\bm{\beta}}\beta^{r}\right)\left(A\underline{\bm{\alpha}}\alpha^{m+1}-B\underline{\bm{\beta}}\beta^{m+1}\right)\\
&\ \ - \left(A\underline{\bm{\alpha}}\alpha^{r+1}-B\underline{\bm{\beta}}\beta^{r+1}\right)\left(A\underline{\bm{\alpha}}\alpha^{m}-B\underline{\bm{\beta}}\beta^{m}\right)\\
&= (-q)^{m} AB(\alpha-\beta) \left(\underline{\bm{\alpha}}\underline{\bm{\beta}}\alpha^{r-m}- \underline{\bm{\beta}}\underline{\bm{\alpha}}\beta^{r-m}\right).
\end{align*}
We require the Eqs. (\ref{eq:1}) and (\ref{eq:2}). Substituting these into the previous equation, we have
\begin{align*}
&\mathbb{H}\mathcal{J}_{r}\mathbb{H}\mathcal{J}_{m+1}-\mathbb{H}\mathcal{J}_{r+1}\mathbb{H}\mathcal{J}_{m}\\
&=\frac{(-q)^{m}}{\alpha-\beta} AB \left\lbrace \begin{array}{c} (\mathbb{H}\mathcal{L}_{0}-(q^{3}+pq-q+1))(\alpha^{r-m}-\beta^{r-m})\\
+q(\alpha-\beta)(\mathbb{H}\mathcal{F}_{0}-\bm{\omega})(\alpha^{r-m}+\beta^{r-m}) \end{array} \right\rbrace \\
&=(-q)^{m} AB \left((\mathbb{H}\mathcal{L}_{0}-(q^{3}+pq-q+1))\mathcal{F}_{r-m}+q(\mathbb{H}\mathcal{F}_{0}-\bm{\omega}) \mathcal{L}_{r-m} \right).
\end{align*}
The second identity in the above equality, can be proved using $\mathcal{L}_{r-m}=\alpha^{r-m}+\beta^{r-m}$ and $\mathcal{F}_{r-m}=\frac{\alpha^{r-m}-\beta^{r-m}}{\alpha-\beta}$. This proof is completed.
\end{proof}

In particular, if $m=r-1$ in this theorem and using the identity $\mathcal{L}_{1}=p$, we obtain Cassini's identities for generalized hybrid Fibonacci numbers. Now, taking $m=r$ in the Theorem \ref{teoD} and using the identities $\mathcal{F}_{0}=0$ and $\mathcal{L}_{0}=2$, we obtain the next identity.
\begin{corollary}
For any integer $r\geq 0$, we have
\begin{equation}\label{eq:7}
\mathbb{H}\mathcal{J}_{r+1}\mathbb{H}\mathcal{J}_{r}-\mathbb{H}\mathcal{J}_{r}\mathbb{H}\mathcal{J}_{r+1}=2(-q)^{r+1} AB (\mathbb{H}\mathcal{F}_{0}-\bm{\omega}),
\end{equation}
where $A=b-a\beta$, $B=b-a\alpha$ and $\bm{\omega}=(1-p)\textbf{i}-q\bm{\varepsilon}+(p^{2}+q+1)\textbf{h}$.
\end{corollary}

After deriving these three famous identities, we present some other identities for the hybrid $(p,q)$-Fibonacci and hybrid $(p,q)$-Lucas numbers.
\begin{theorem}
For any integers $n$, $r$ and $s$, we have
\begin{equation}\label{equa:7}
\mathbb{H}\mathcal{L}_{n+r}\mathbb{H}\mathcal{F}_{n+s}-\mathbb{H}\mathcal{L}_{n+s}\mathbb{H}\mathcal{F}_{n+r}=2(-q)^{n+r}\mathcal{F}_{s-r}(\mathbb{H}\mathcal{L}_{0}-(q^{3}+pq-q+1)).
\end{equation}
\end{theorem}
\begin{proof}
The Binet formulas for the hybrid $(p,q)$-Lucas and hybrid $(p,q)$-Fibonacci numbers give
\begin{align*}
(\alpha-\beta)(\mathbb{H}\mathcal{L}_{n+r}\mathbb{H}\mathcal{F}_{n+s}&-\mathbb{H}\mathcal{L}_{n+s}\mathbb{H}\mathcal{F}_{n+r})\\
&=\left(\underline{\bm{\alpha}}\alpha^{n+r}+\underline{\bm{\beta}}\beta^{n+r}\right)\left(\underline{\bm{\alpha}}\alpha^{n+s}-\underline{\bm{\beta}}\beta^{n+s}\right)\\
&\ \ -\left(\underline{\bm{\alpha}}\alpha^{n+s}+\underline{\bm{\beta}}\beta^{n+s}\right)\left(\underline{\bm{\alpha}}\alpha^{n+r}-\underline{\bm{\beta}}\beta^{n+r}\right)\\
&=(\alpha \beta)^{n}(\alpha^{s}\beta^{r}-\alpha^{r}\beta^{s})(\underline{\bm{\alpha}}\underline{\bm{\beta}}+\underline{\bm{\beta}}\underline{\bm{\alpha}}).
\end{align*}
Using Eqs. (\ref{eq:1}) and (\ref{eq:2}), we have
$$\mathbb{H}\mathcal{L}_{n+r}\mathbb{H}\mathcal{F}_{n+s}-\mathbb{H}\mathcal{L}_{n+s}\mathbb{H}\mathcal{F}_{n+r}=2(-q)^{n+r}\mathcal{F}_{s-r}(\mathbb{H}\mathcal{L}_{0}-(q^{3}+pq-q+1)).$$ The proof is completed.
\end{proof}

After deriving these famous identities, we present some other identities for the generalized hybrid Fibonacci numbers. In particular, when using the Binet formulas to obtain identities for the hybrid $(p,q)$-Fibonacci and hybrid $(p,q)$-Lucas numbers, we require $\underline{\bm{\alpha}}^{2}$ and $\underline{\bm{\beta}}^{2}$. These products are given in the next lemma.
\begin{lemma}
We have
\begin{equation}\label{eq:8}
\underline{\bm{\alpha}}^{2}=(\mathbb{H}\mathcal{L}_{0}+r_{p,q})+(\alpha-\beta)(\mathbb{H}\mathcal{F}_{0}+s_{p,q})
\end{equation}
and
\begin{equation}\label{eq:9}
\underline{\bm{\beta}}^{2}=(\mathbb{H}\mathcal{L}_{0}+r_{p,q})-(\alpha-\beta)(\mathbb{H}\mathcal{F}_{0}+s_{p,q}),
\end{equation}
where $r_{p,q}=-1+\frac{p}{2}(\mathcal{F}_{6}+2\mathcal{F}_{3}-\mathcal{F}_{2})+q(\mathcal{F}_{5}+2\mathcal{F}_{2}-\mathcal{F}_{1})$, $s_{p,q}=\frac{1}{2}(\mathcal{F}_{6}+2\mathcal{F}_{3}-\mathcal{F}_{2})$ and $\mathcal{F}_{n}$ is the $n$-th $(p,q)$-Fibonacci number.
\end{lemma}
\begin{proof}
From the definition of $\underline{\bm{\alpha}}$ and using $\textbf{i}^{2}=-1, \bm{\varepsilon}^{2}=0, \textbf{h}^{2}=1$, $ \textbf{ih}=-\textbf{hi}=\bm{\varepsilon}+\textbf{i}$ in Table \ref{table:1} and $\alpha^{n}=\mathcal{F}_{n}\alpha+q\mathcal{F}_{n-1}$ for $n\geq1$, we have
\begin{align*}
\underline{\bm{\alpha}}^{2}&=(1+\alpha \textbf{i}+\alpha^{2}\bm{\varepsilon}+\alpha^{3}\textbf{h})(1+\alpha \textbf{i}+\alpha^{2}\bm{\varepsilon}+\alpha^{3}\textbf{h})\\
&=2(1+\alpha \textbf{i}+\alpha^{2}\bm{\varepsilon}+\alpha^{3}\textbf{h})+(\alpha^{6}+2\alpha^{3}-\alpha^{2}-1)\\
&=2+2\alpha \textbf{i}+(2p\alpha+2q)\bm{\varepsilon}+((2p^{2}+2q)\alpha+2pq)\textbf{h})+(\alpha^{6}+2\alpha^{3}-\alpha^{2}-1)\\
&=2+p\textbf{i}+(p^{2}+2q)\bm{\varepsilon}+(p^{3}+3pq)\textbf{h}+(\alpha-\beta)(\textbf{i}+p\bm{\varepsilon}+(p^{2}+q)\textbf{h})\\
&\ \ +((\mathcal{F}_{6}\alpha+q\mathcal{F}_{5})+2(\mathcal{F}_{3}\alpha+q\mathcal{F}_{2})-(\mathcal{F}_{2}\alpha+q\mathcal{F}_{1})-1)\\
&=(\mathbb{H}\mathcal{L}_{0}+r_{p,q})+(\alpha-\beta)(\mathbb{H}\mathcal{F}_{0}+s_{p,q}),
\end{align*}
where $r_{p,q}=-1+\frac{p}{2}(\mathcal{F}_{6}+2\mathcal{F}_{3}-\mathcal{F}_{2})+q(\mathcal{F}_{5}+2\mathcal{F}_{2}-\mathcal{F}_{1})$ and $s_{p,q}=\frac{1}{2}(\mathcal{F}_{6}+2\mathcal{F}_{3}-\mathcal{F}_{2})$ and the final equation gives Eq. (\ref{eq:8}). The other can be computed similarly.
\end{proof}

We present some interesting identities for hybrid $(p,q)$-Fibonacci, hybrid $(p,q)$-Lucas numbers and generalized hybrid Fibonacci numbers.
\begin{theorem}
For any integer $n\geq 0$, we have
\begin{equation}\label{equa:9}
\mathbb{H}\mathcal{L}_{n}^{2}-\mathbb{H}\mathcal{F}_{n}^{2}=\left\lbrace 
\begin{array}{c}
\frac{p^{2}+4q-1}{p^{2}+4q}(\mathbb{H}\mathcal{L}_{0}+r_{p,q})\mathcal{L}_{2n}+(\mathbb{H}\mathcal{F}_{0}+s_{p,q})\mathcal{F}_{2n}  \\ 
+2\frac{(p^{2}+4q+1)(-q)^{n}}{p^{2}+4q}(\mathbb{H}\mathcal{L}_{0}-(q^{3}+pq-q+1)).
\end{array}%
\right\rbrace
\end{equation}
\end{theorem}
\begin{proof}
Using the Binet formulas for the hybrid $(p,q)$-Fibonacci and hybrid $(p,q)$-Lucas numbers, we obtain
\begin{align*}
(p^{2}+4q)(\mathbb{H}\mathcal{L}_{n}^{2}-\mathbb{H}\mathcal{F}_{n}^{2})&=(p^{2}+4q)\left(\underline{\bm{\alpha}}\alpha^{n}+\underline{\bm{\beta}}\beta^{n}\right)^{2}-\left(\underline{\bm{\alpha}}\alpha^{n}-\underline{\bm{\beta}}\beta^{n}\right)^{2}\\
&=(p^{2}+4q-1)(\underline{\bm{\alpha}}^{2}\alpha^{2n}+\underline{\bm{\beta}}^{2}\beta^{2n})\\
&\ \ +(p^{2}+4q+1)(\alpha \beta)^{n}(\underline{\bm{\alpha}}\underline{\bm{\beta}}+\underline{\bm{\beta}}\underline{\bm{\alpha}}).
\end{align*}
Substituting Eqs. (\ref{eq:1}) and (\ref{eq:2}) into the last equation, we have
\begin{equation}\label{eq:10}
\begin{aligned}
(p^{2}+4q)(\mathbb{H}\mathcal{L}_{n}^{2}-\mathbb{H}\mathcal{F}_{n}^{2})&=(p^{2}+4q-1)(\underline{\bm{\alpha}}^{2}\alpha^{2n}+\underline{\bm{\beta}}^{2}\beta^{2n})\\
&\ \ +2(p^{2}+4q+1)(\alpha \beta)^{n}(\mathbb{H}\mathcal{L}_{0}-(q^{3}+pq-q+1)).
\end{aligned}
\end{equation}
Then, using Eqs. (\ref{eq:8}) and (\ref{eq:9}), we obtain 
\begin{equation}\label{eq:11}
\begin{aligned}
\underline{\bm{\alpha}}^{2}\alpha^{2n}+\underline{\bm{\beta}}^{2}\beta^{2n}&=(\alpha^{2n}+\beta^{2n})(\mathbb{H}\mathcal{L}_{0}+r_{p,q})\\
&\ \ +(\alpha-\beta) (\mathbb{H}\mathcal{F}_{0}+s_{p,q})(\alpha^{2n}-\beta^{2n}).
\end{aligned}
\end{equation}
Substituting Eq. (\ref{eq:11}) into Eq. (\ref{eq:10}) gives Eq. (\ref{equa:9}).
\end{proof}

\begin{theorem}\label{teoF}
For any integers $m\geq n\geq 0$, we have
\begin{equation}\label{equa:12}
\mathbb{H}\mathcal{F}_{n}\mathbb{H}\mathcal{J}_{m}-\mathbb{H}\mathcal{J}_{m}\mathbb{H}\mathcal{F}_{n}=2(-q)^{n+1} \mathcal{J}_{m-n}(\mathbb{H}\mathcal{F}_{0}-\bm{\omega}),
\end{equation}
where $\bm{\omega}=(1-p)\textbf{i}-q\bm{\varepsilon}+(p^{2}+q+1)\textbf{h}$ and $\mathcal{J}_{n}=\frac{A\alpha^{n}-B\beta^{n}}{\alpha-\beta}$ is the $n$-th generalized Fibonacci number.
\end{theorem}
\begin{proof}
The Binet formulas for the hybrid $(p,q)$-Fibonacci and generalized hybrid Fibonacci numbers give
\begin{align*}
(p^{2}+4q)(\mathbb{H}\mathcal{F}_{n}\mathbb{H}\mathcal{J}_{m}-\mathbb{H}\mathcal{J}_{m}\mathbb{H}\mathcal{F}_{n})&=\left(\underline{\bm{\alpha}}\alpha^{n}-\underline{\bm{\beta}}\beta^{n}\right)\left(A\underline{\bm{\alpha}}\alpha^{m}-B\underline{\bm{\beta}}\beta^{m}\right)\\
&\ \ -\left(A\underline{\bm{\alpha}}\alpha^{m}-B\underline{\bm{\beta}}\beta^{m}\right)\left(\underline{\bm{\alpha}}\alpha^{n}-\underline{\bm{\beta}}\beta^{n}\right)\\
&=(A\alpha^{m}\beta^{n}-B\alpha^{n}\beta^{m})(\underline{\bm{\alpha}}\underline{\bm{\beta}}-\underline{\bm{\beta}}\underline{\bm{\alpha}}).
\end{align*}
Using Eqs. (\ref{eq:1}) and (\ref{eq:2}), we have
\begin{align*}
\mathbb{H}\mathcal{F}_{n}\mathbb{H}\mathcal{J}_{m}-\mathbb{H}\mathcal{J}_{m}\mathbb{H}\mathcal{F}_{n}&=\frac{2q(\alpha \beta)^{n}}{p^{2}+4q}(A\alpha^{m-n}-B\beta^{m-n})(\alpha-\beta)(\mathbb{H}\mathcal{F}_{0}-\bm{\omega})\\
&=-2(-q)^{n+1} \mathcal{J}_{m-n}(\mathbb{H}\mathcal{F}_{0}-\bm{\omega}),
\end{align*}
where $\bm{\omega}=(1-p)\textbf{i}-q\bm{\varepsilon}+(p^{2}+q+1)\textbf{h}$ and $\mathcal{J}_{n}$ is the $n$-th generalized Fibonacci number defined by $\mathcal{J}_{n}=\frac{A\alpha^{n}-B\beta^{n}}{\alpha-\beta}$. So, the theorem is proved.
\end{proof}

Taking $m=n$ in the Theorem \ref{teoF} and using $\mathcal{J}_{0}=a$, we obtain the next identity.
\begin{corollary}
For any integer $n\geq 0$, we have
\begin{equation}\label{eq:13}
\mathbb{H}\mathcal{F}_{n}\mathbb{H}\mathcal{J}_{n}-\mathbb{H}\mathcal{J}_{n}\mathbb{H}\mathcal{F}_{n}=2a(-q)^{n+1}(\mathbb{H}\mathcal{F}_{0}-\bm{\omega}),
\end{equation}
where $A=b-a\beta$, $B=b-a\alpha$ and $\bm{\omega}=(1-p)\textbf{i}-q\bm{\varepsilon}+(p^{2}+q+1)\textbf{h}$.
\end{corollary}

\section{Conclusions}
There are differences between the hybrid generalized Fibonacci numbers and the coefficient generalized Fibonacci quaternions. For example, the usual coefficient generalized Fibonacci quaternionic units are $\textbf{i}^{2}=\textbf{j}^{2}=\textbf{k}^{2}=\textbf{ijk}^{2}=-1$ whereas the hybrid generalized Fibonacci quaternionic units are $\textbf{i}^{2}=-1, \bm{\varepsilon}^{2}=0, \textbf{h}^{2}=1$ and $\textbf{ih}=-\textbf{hi}=\bm{\varepsilon}+\textbf{i}$. 

In this work, we have examined a new type of numbers, which are non-commutative. We named this number set as generalized hybrid Fibonacci numbers because it is a linear combination of well-known complex, hyperbolic and dual Fibonacci numbers. We have given the relation $\textbf{ih}=-\textbf{hi} =\bm{\varepsilon}+\textbf{i}$ between the units $\{\textbf{i}, \bm{\varepsilon}, \textbf{h}\}$ of these three number systems, and we have seen the algebraic consistency of this relation. Thus, we have obtained some properties of the generalized hybrid Fibonacci numbers.


\end{document}